\documentclass{amsart}

\usepackage{graphicx}         
\usepackage[tight]{subfigure} 
\usepackage{url}              

\newtheorem{theorem}{Theorem}[section]
\newtheorem{corollary}[theorem]{Corollary}
\newtheorem{lemma}[theorem]{Lemma}

\newtheorem{construction}[theorem]{Construction}
\newtheorem{problem}[theorem]{Problem}

\theoremstyle{definition}
\newtheorem{defn}[theorem]{Definition}
\newtheorem{example}[theorem]{Example}
\newtheorem*{notation}{Notation}

\newcommand{\adj}{\mathop{adj}}
\newcommand{\C}{\mathbb{C}}
\newcommand{\col}{\mathop{col}}
\newcommand{\co}{\colon\thinspace}
\newcommand{\dual}{\mathcal{D}}
\newcommand{\hasse}{\mathcal{H}}
\newcommand{\of}{\!:\!}
\newcommand{\inc}{\mathop{inc}}
\newcommand{\isarc}{\mathop{is\_arc}}
\newcommand{\iscol}{\mathop{is\_col}}
\newcommand{\isface}{\mathop{is\_face}}
\newcommand{\isnode}{\mathop{is\_node}}

\newcommand{\N}{\mathbb{N}}

\newcommand{\tri}{\mathcal{T}}
\newcommand{\tw}{\mathrm{tw}}
\newcommand{\Z}{\mathbb{Z}}

\begin{document}

\title[Courcelle's theorem for triangulations]{Courcelle's theorem
    for triangulations}
\author{Benjamin A.\ Burton}
\address{School of Mathematics and Physics \\
    The University of Queensland \\
    Brisbane QLD 4072 \\
    Australia}
\email{bab@maths.uq.edu.au}
\author{Rodney G. Downey}
\address{School of Mathematics, Statistics and Operations Research \\
    Victoria University, New Zealand.}
\email{Rod.Downey@msor.vuw.ac.nz}
\thanks{The first author is supported by the Australian Research Council
    under the Discovery Projects funding scheme
    (projects DP1094516, DP110101104),
    and the second author is supported by the Marsden fund of New Zealand.}

\begin{abstract}
    In graph theory, Courcelle's theorem essentially states
    that, if an algorithmic problem can be formulated in
    monadic second-order logic, then it can be solved in linear time
    for graphs of bounded treewidth.
    We prove such a metatheorem for a general class of triangulations
    of arbitrary fixed dimension~$d$,
    including all triangulated $d$-manifolds:
    if an algorithmic problem can be expressed in
    monadic second-order logic, then it can be solved in linear time for
    triangulations whose dual graphs have bounded treewidth.

    We apply our results to 3-manifold topology, a setting
    with many difficult computational problems but very few parameterised
    complexity results,
    and where treewidth has practical relevance as a parameter.
    Using our metatheorem, we recover and generalise earlier
    fixed-parameter tractability results on taut angle structures
    and discrete Morse theory respectively,
    and prove a new fixed-parameter tractability
    result for computing the powerful but complex Turaev-Viro invariants
    on 3-manifolds.
\end{abstract}

\maketitle

%
%

\section{Introduction}

Parameterised complexity is a relatively new and highly successful
framework for understanding the computational complexity of ``hard'' problems
for which we do not have a polynomial-time algorithm \cite{downey99-param}.
The key idea is to measure the complexity not just in terms of the input
size (the traditional approach), but also in terms of
additional \emph{parameters} of the input or of the problem itself.
The result is that, even if a problem is (for instance) NP-hard,
we gain a richer theoretical understanding of those classes of inputs for
which the problem is still tractable, and we acquire new practical tools
for solving the problem in real software.

For example,
finding a Hamiltonian cycle in an arbitrary graph is NP-complete,
but for graphs of fixed treewidth $\leq k$
it can be solved in linear time in the input size
\cite{downey99-param}.  
In general, a problem is called \emph{fixed-parameter tractable} in the
parameter $k$ if, for any class of inputs where $k$ is universally
bounded, the running time becomes polynomial in the input size.

Treewidth in particular
(which roughly measures how ``tree-like'' a graph is
\cite{robertson86-algorithmic})
is extremely useful as a parameter.
A great many graph problems are known to
be fixed-parameter tractable in the treewidth,
in a large part due to Courcelle's celebrated ``metatheorem''
\cite{courcelle87-context-free,courcelle90-rewriting}:
for \emph{any} decision problem $P$ on graphs,
if $P$ can be framed using monadic second-order logic,
then $P$ can be solved in \emph{linear time} for graphs of
universally bounded treewidth $\leq k$.

The motivation behind this paper is to develop the tools of
parameterised complexity for systematic use in the field of
geometric topology, and in particular for 3-manifold topology.
This is a field with natural and fundamental algorithmic problems,
such as determining whether two knots or two triangulations
are topologically equivalent,
and in three dimensions such problems are often decidable but
extremely complex \cite{matveev03-algms}.

Parameterised complexity is appealing as a theoretical
framework for identifying when ``hard'' topological problems can be
solved quickly.  Unlike average-case complexity or
generic complexity, it avoids the need to work with
\emph{random inputs}---something that still poses major difficulties
for 3-manifold topology \cite{dunfield06-random-covers}.
The viability of this framework is shown by recent
parameterised complexity results in topological settings such as
knot polynomials \cite{makowsky05-tutte,makowsky03-knot},
angle structures \cite{burton13-taut},
discrete Morse theory \cite{burton13-morse}, and the
enumeration of 3-manifold triangulations \cite{burton14-fpt-enum}.

The treewidth parameter plays a key role in all of the aforementioned
results.  For topological problems whose input is a triangulation $\tri$,
we measure the treewidth of the \emph{dual graph} $\dual(\tri)$,
whose nodes describe top-dimensional simplices of
$\tri$, and whose arcs show how these simplices are joined together
along their facets.
In 3-manifold topology this parameter has a natural interpretation,
and there are common settings in which the treewidth remains small; see
Section~\ref{s-app} for details.

Our main result in this paper is a Courcelle-like metatheorem
for use with triangulations.  Specifically, in Section~\ref{s-meta}
we describe a form of monadic second-order
logic for use with triangulations of fixed dimension $d$, and
we show that all problems expressible in this
logical framework are fixed-parameter tractable in the treewidth of the
dual graph of the input triangulation (Theorem~\ref{t-tri}).

Section~\ref{s-app} gives several applications of this metatheorem.
We recover earlier results on
taut angle structures \cite{burton13-taut} and discrete Morse theory
\cite{burton13-morse}, generalise the latter result
to arbitrary dimension (Theorem~\ref{t-morse}),
and prove a new result on computing the
Turaev-Viro invariants of 3-manifolds (Theorem~\ref{t-tv}).
These new results on discrete Morse theory and Turaev-Viro invariants
have significant practical potential;
see Section~\ref{s-app} for further discussion.

We prove our main result in two stages.
In Section~\ref{s-coloured} we translate several variants of
Courcelle's theorem from simple graphs to the more flexible
setting of edge-coloured graphs.
In Section~\ref{s-meta}
we show how to use an edge-coloured graph to encode the
full structure of a triangulation (using a coloured
variant of the well-known \emph{Hasse diagram}),
and using this we translate the variants of Courcelle's theorem
up to the final setting of triangulations.

We emphasise that our results
depend crucially on how we define a triangulation.
We do not allow arbitrary simplicial complexes, where low-dimensional
faces can be joined or ``pinched'' together independently of
the larger simplices to which they belong.
Instead our triangulations are formed purely by joining together
$d$-simplices along their $(d-1)$-dimensional facets.
This definition is flexible enough to encompass any reasonable
concept of a triangulated $d$-manifold.  Indeed, it covers
structures more general than simplicial complexes,
such as the highly efficient
\emph{one-vertex triangulations} \cite{jaco03-0-efficiency}
and \emph{ideal triangulations} \cite{thurston78-lectures}
favoured by many 3-manifold topologists.

%
%

\newpage 

\section{Preliminaries}

\subsection{Treewidth} \label{s-tw}

Throughout this paper we work with several common classes of graphs; we briefly outline these classes here.
We use the terms \emph{node} and \emph{arc} when working with graphs to avoid confusion with the vertices and edges
of triangulations.

A \emph{simple graph} $G=(V,E)$ is a finite set $V$ of nodes and a
finite set $E$ of arcs, where each arc is an unordered pair $\{v,w\}$ of
distinct nodes $v,w \in V$ (so we do not allow loops that join a node
with itself, or multiple arcs between the same two nodes).  A
\emph{multigraph} $G=(V,E)$ is a finite set $V$ of nodes and a finite
multiset of arcs, where each arc is an unordered pair $\{v,w\}$ of nodes
$v,w \in V$ and we allow $v=w$ (so loops and multiple arcs are allowed).
An \emph{edge-coloured graph} $G=(V,E,C)$ is a finite set $V$ of nodes,
a finite set $C$ of \emph{colours} and a finite set $E$ of arcs, where
each arc is a pair $(\{v,w\},c)$ with $v,w \in V$, $c \in C$ and $v \ne
w$ (i.e., each arc is given a colour, loops are not allowed, and
multiple arcs between the same two nodes are only allowed if they are
assigned different colours).  Unless otherwise specified, any general
statement about \emph{graphs} refers to all of these classes.

For any graph $G$, the \emph{size} of the graph is denoted by $|G|$.
The size counts the total number of nodes and arcs, i.e., $|G|=|V|+|E|$.

The treewidth of a graph $G$, introduced by Robertson and Seymour
\cite{robertson86-algorithmic}, essentially measures how far $G$ is from
being a tree: any tree will have treewidth $1$ (the smallest possible),
and a complete graph will have treewidth $|V|-1$ (the largest possible
for a given $V$).  Graphs of small treewidth are often easier to work
with, as Courcelle's theorem (described below) so strikingly shows.
The full definition is as follows.

Given a simple graph or multigraph $G=(V,E)$, a \emph{tree decomposition} of
$G$ consists of a (finite) tree $T$ and \emph{bags} $B_\tau \subseteq V$ for
each node $\tau$ of $T$ that satisfy the following constraints:
\begin{itemize}
    \item each $v \in V$ belongs to some bag $B_\tau$;
    \item for each arc of $G$, its two endpoints $v,w$ belong to some
    common bag $B_\tau$;
    \item for each $v \in V$, the bags containing $v$ correspond to a
    connected subtree of $T$.
\end{itemize}
The \emph{width} of this tree decomposition is $\max |B_\tau|-1$, and the
\emph{treewidth} of $G$
is the smallest width of any tree decomposition of $G$, which
we denote by $\tw(G)$.
Figure~\ref{fig-tw} illustrates a simple graph $G$
and a corresponding tree decomposition of width~2.

\begin{figure}[tb]
    \centering
    \subfigure[A simple graph $G$\label{fig-tw-graph}]{%
        \includegraphics[scale=0.6]{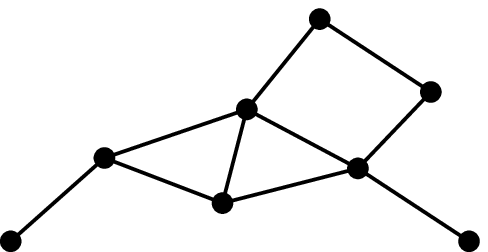}}
    \quad
    \subfigure[The bags of a tree decomposition\label{fig-tw-bags}]{%
        \qquad\includegraphics[scale=1.0]{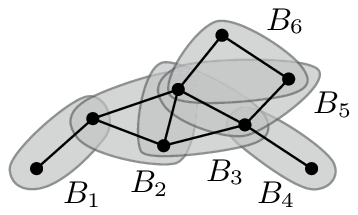}\qquad}
    \quad
    \subfigure[The underlying tree\label{fig-tw-tree}]{%
        \quad\includegraphics[scale=1.0]{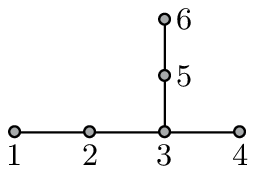}\quad}
    \caption{A simple graph and a corresponding tree decomposition}
    \label{fig-tw}
\end{figure}

Although computing treewidth is NP-complete, it is also fixed-parameter
tractable:
\begin{theorem}[Bodlaender \cite{bodlaender96-linear}]
    \label{t-tree-decomp}
    There exists a computable function $f$
    and an algorithm which, given a simple graph $G = (V,E)$ of
    treewidth $k = \tw(G)$, can compute a tree decomposition of width
    $k$ in time $f(k) \cdot |V|$.
\end{theorem}

More precisely, we have $f(k) \in 2^{k^{O(1)}}$; see
\cite{bodlaender96-linear,flum06} for details.


\subsection{Monadic second-order logic and Courcelle's theorem}
\label{s-msol}

Monadic second-order logic, or MSO logic, is our framework for making
statements about graphs.  Here we give a brief overview of the key
concepts as they appear in the context of simple graphs;
see a standard text such as \cite{flum06} for further details.
What we describe here is sometimes called \emph{extended} MSO logic,
or $\mathit{MS}_2$ logic;
this highlights the fact that we can access arcs directly through variables
and sets, and not just indirectly through a binary relation on nodes.

MSO logic supports:
\begin{itemize}
    \item all of the standard boolean operations of
    propositional logic: $\wedge$ (and), $\vee$ (or), $\neg$ (negation),
    $\rightarrow$ (implication), and so on;

    \item variables to represent nodes, arcs, sets of nodes,
    or sets of arcs of a graph;

    \item the standard quantifiers from first-order logic:
    $\forall$ (the universal quantifier), and
    $\exists$ (the existential quantifier),
    which may be applied to any of these variable types;

    \item the binary equality relation $=$,
    which can be applied to nodes, arcs, sets of nodes, or sets of arcs;
    \item the binary inclusion relation $\in$,
    which can relate nodes to sets of nodes, or arcs to sets of arcs;
    \item the binary incidence relation $\inc(e,v)$,
    which encodes the fact that $e$ is an arc, $v$ is a node,
    and $v$ is one of the two endpoints of $e$;
    \item the binary adjacency relation $\adj(v,v')$,
    which encodes the fact that $v$ and $v'$ are the two endpoints
    of some common arc.
\end{itemize}

By convention, we use lower-case letters
$v,w,\ldots$ to represent nodes and
$e,f,\ldots$ to represent arcs,
and upper-case letters
$V,W,\ldots$ to represent sets of nodes and
$E,F,\ldots$ to represent sets of arcs.

Note that sets are simply a convenient representation of unary relations
on nodes and arcs.  A distinguishing feature of MSO logic is that
we can quantify over unary relations (i.e., we can quantify
over set variables as outlined above).

We use the notation $\phi(x_1,\ldots,x_t)$ to denote an MSO formula
with $t$ free variables (i.e., variables not bound by
$\forall$ or $\exists$ quantifiers).
An \emph{MSO sentence} is an MSO formula with no free variables at all.
If $G$ is a simple graph and $\phi$ is a MSO sentence,
we use the notation $G \models \phi$ to indicate that the interpretation
of $\phi$ in the graph $G$ is a true statement.

We give three variants of Courcelle's theorem, which relate to
algorithms for (i)~\emph{decision} problems on graphs;
(ii)~\emph{optimising} quantities on graphs; and
(iii)~\emph{evaluating} functions on graphs.
Each variant works with MSO formulae in different ways, which we
now outline in turn.

\subsubsection{Decision problems}

\begin{example}[3-colourability]
    The following MSO sentence expresses the fact that a simple graph is
    3-colourable:

    \medskip

    \begin{tabular}{l}
        $\exists V_1 \exists V_2 \exists V_3 \ \forall v \forall w$ \\
        $(v \in V_1 \vee v \in V_2 \vee v \in V_3) \wedge {}$ \\
        $\neg ((v \in V_1 \wedge v \in V_2) \vee
               (v \in V_2 \wedge v \in V_3) \vee
               (v \in V_1 \wedge v \in V_3)) \wedge {}$ \\
        $\adj(v,w) \rightarrow
            \neg ((v \in V_1 \wedge w \in V_1) \vee
                  (v \in V_2 \wedge w \in V_2) \vee
                  (v \in V_3 \wedge w \in V_3)).$
    \end{tabular}

    \medskip

    The sets $V_1,V_2,V_3$ indicate which nodes are assigned each of the
    three available colours.
    The second and third lines ensure that $V_1$, $V_2$ and $V_3$ partition
    the nodes, and the final line ensures that any two adjacent nodes
    are coloured differently.
\end{example}

\begin{theorem}[Courcelle
    \cite{courcelle87-context-free,courcelle90-rewriting}]
    Given a simple graph $G=(V,E)$, its treewidth $k=\tw(G)$ and a fixed
    MSO sentence $\phi$, there exists a computable function $f$
    and an algorithm for testing
    whether $G \models \phi$ that runs in time $f(k, |\phi|) \cdot |G|$.
\end{theorem}

\begin{corollary} \label{c-courcelle-test}
    For any fixed MSO sentence $\phi$ and
    any class $K$ of simple graphs with universally bounded treewidth,
    it is possible to test whether $G \models \phi$
    for graphs $G \in K$ in time $O(|G|)$.
\end{corollary}

In other words, testing for $\phi$ is linear-time fixed-parameter tractable
in the treewidth.
Note that, thanks to Theorem~\ref{t-tree-decomp},
we do not need to supply an explicit tree decomposition of the input graph $G$
in advance.

For the remainder of this paper we will formulate our results using the
language of Corollary~\ref{c-courcelle-test}, omitting explicit
references to the underlying function $f$.

\subsubsection{Optimisation problems}
\label{s-extremum}

A \emph{restricted MSO extremum problem} consists of
an MSO formula $\phi(A_1,\ldots,A_t)$
with free set variables $A_1,\ldots,A_t$,
and a rational linear function $g(x_1,\ldots,x_t)$.  Its interpretation
is as follows: given a simple graph $G$ as input,
we are asked to minimise $g(|A_1|,\ldots,|A_t|)$ over all sets
$A_1,\ldots,A_t$ for which $G \models \phi(A_1,\ldots,A_t)$,
where $|A_i|$ as usual denotes the number of objects in the set $A_i$.

\begin{example}[Dominating set]
    The well-known problem \emph{dominating set}
    asks for the smallest set of nodes $D$ in a given graph $G$ for which
    every node in $G$ is either in $D$ or adjacent to some node in $D$.

    To formulate this as a restricted MSO extremum problem,
    we use a single free set variable $D$,
    and minimise the linear function
    $g(|D|) = |D|$ under the following MSO constraint:

    \medskip

    \begin{tabular}{l}
        $\forall v\ \exists w\ (v \in D) \vee (w \in D \wedge \adj(v,w)).$
    \end{tabular}
\end{example}

Courcelle's theorem has been extended to work with such problems:

\begin{theorem}[Arnborg, Lagergren and Seese \cite{arnborg91-easy}]
    \label{t-courcelle-opt}
    For any restricted MSO extremum problem $P$ and
    any class $K$ of simple graphs with universally bounded treewidth,
    it is possible to solve $P$ for graphs $G \in K$ in time $O(|G|)$
    under the uniform cost measure.
\end{theorem}

Recall that the \emph{uniform cost measure}
assumes that elementary arithmetic operations run in constant time;
see a standard text such as \cite{aho75} for details.
Again, Theorem~\ref{t-tree-decomp} ensures that
we do not need to supply an explicit tree decomposition in advance.

Arnborg et~al.\ \cite{arnborg91-easy} prove more general results:
for instance, they allow evaluation functions on weighted graphs
(where integer or rational weights on the nodes and/or arcs are supplied
with the problem instance), they allow additional constants to be given
with the problem instance, and they discuss non-linear extremum problems
and enumeration problems.
For simplicity we restrict our attention here to the restricted class
of extremum problems as described above.

\subsubsection{Evaluation problems}
\label{s-evaluation}

We move now to evaluation problems, which
are generalised counting problems: essentially,
we assign a value to each solution to some MSO-defined
problem on a graph, and then sum these values over all solutions.
Evaluation problems come in both additive or multiplicative variants,
and counting problems correspond to multiplicative variants in which
all solutions have value~$1$.

More precisely, an \emph{MSO evaluation problem} is defined as follows.
The problem consists of an MSO formula $\phi(A_1,\ldots,A_t)$
with $t$ free set variables $A_1,\ldots,A_t$.
The input to the problem is a simple graph $G=(V,E)$,
together with $t$ weight functions
$w_1,\ldots,w_t \co V \sqcup E \to R$ on nodes and/or arcs,
where $R$ is some ring or field.
The problem then asks us to compute one of the quantities
\[ \sum_{G \models \phi(A_1,\ldots,A_t)}
    \left\{
    \sum_{i=1}^t \sum_{x_i \in A_i} w_i(x_i)
    \right\}
    \quad\mbox{or}\quad
   \sum_{G \models \phi(A_1,\ldots,A_t)}
    \left\{
    \prod_{i=1}^t \prod_{x_i \in A_i} w_i(x_i)
    \right\}; \]
we refer to these two variants as \emph{additive} and
\emph{multiplicative} evaluation problems respectively.
For both problems, the outermost sum is over all solutions
$A_1,\ldots,A_t$ that satisfy the MSO formula $\phi$ on the graph $G$.

We now come to our third variant of Courcelle's theorem, which extends
earlier work of Courcelle and Mosbah \cite{courcelle93-monadic}
and Arnborg et~al.\ \cite{arnborg91-easy}.

\begin{theorem}[Courcelle, Makowsky and Rotics \cite{courcelle01-enumeration}]
    \label{t-courcelle-eval}
    For any MSO evaluation problem $P$ and
    any class $K$ of simple graphs with universally bounded treewidth,
    it is possible to solve $P$ for graphs $G \in K$ in time $O(|G|)$
    under the uniform cost measure.
\end{theorem}

Here we interpret the uniform cost measure to allow constant-time
arithmetic operations over the ring or field $R$.
Courcelle et~al.\ \cite{courcelle01-enumeration}
only prove this explicitly for $t=1$ free variable only;
however, they note that the generalisation to a sequence of
$t$ free variables (for fixed $t$) is obvious.


\subsection{Triangulations}
\label{s-prelim-tri}

We now describe the general class of $d$-dimensional triangulations
upon which our metatheorem operates.
In essence, these triangulations are formed by identifying (or ``gluing'')
facets of $d$-simplices in pairs.
This definition does not cover all simplicial complexes (in which
lower-dimensional faces can also be identified independently), but it does
encompass any reasonable definition of a triangulated $d$-manifold;
moreover, it allows more general structures that simplicial complexes
do not, such as the highly efficient
``1-vertex triangulations'' and ``ideal triangulations''
that are often found in algorithmic 3-manifold topology
\cite{jaco03-0-efficiency,thurston78-lectures}.
The details follow.

Let $d \in \N$.  A \emph{$d$-dimensional triangulation}
consists of a collection of abstract $d$-simplices $\Delta_1,\ldots,\Delta_n$,
some or all of whose facets\footnote{%
    Recall that a \emph{facet} of a $d$-simplex is a $(d-1)$-dimensional face.}
are affinely identified (or ``glued'') in pairs.
Each facet $F$ of a $d$-simplex may only be identified with at most one
other facet $F'$ of a $d$-simplex; this may be another facet
of the \emph{same} $d$-simplex, but it cannot be $F$ itself.
Those facets that are not identified with any other facet together form
the \emph{boundary} of the triangulation.

Consider any integer $i$ with $0 \leq i < d$.
There are $\binom{d+1}{i+1}$ distinct $i$-faces of each simplex
$\Delta_1,\ldots,\Delta_n$.
As a consequence of the facet identifications, some of these $i$-faces
become identified with each other; we refer to each class of
identified $i$-faces as a single \emph{$i$-face of the triangulation}.
As usual, 0-faces and 1-faces are called \emph{vertices} and
\emph{edges} respectively.
A \emph{simplex of the triangulation} explicitly refers to one of the
$d$-simplices $\Delta_1,\ldots,\Delta_n$ (not a smaller-dimensional face),
and for convenience we also refer to these as
\emph{$d$-faces of the triangulation}.

A \emph{$d$-manifold triangulation} is simply a $d$-dimensional triangulation
whose underlying topological space is a $d$-manifold when using the
quotient topology.

By convention, we label the vertices of each simplex as $0,\ldots,d$.
We also arbitrarily label the vertices of each $i$-face of the
triangulation as $0,\ldots,i$ (so, for instance, for $i=1$ this corresponds
to placing an arbitrary direction on each edge).
Note that there are many possible ways in which the
vertex labels of an $i$-face of the triangulation
might correspond to vertex labels on the constituent simplices.

\begin{figure}[tb]
    \subfigure[A Klein bottle $\mathcal{K}$\label{fig-kb}]{%
        \includegraphics[scale=0.9]{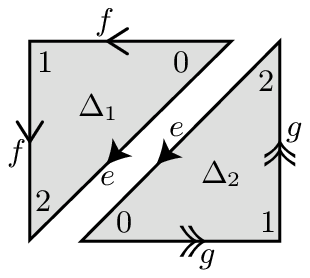}}
    \quad
    \subfigure[A one-tetrahedron solid torus\label{fig-lst}]{%
        \quad\qquad\includegraphics[scale=0.9]{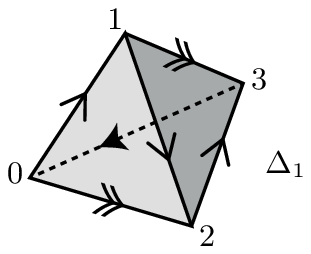}\qquad\quad}
    \quad
    \subfigure[The dual graph $\dual(\mathcal{K})$\label{fig-dual}]{%
        \qquad\includegraphics[scale=0.9]{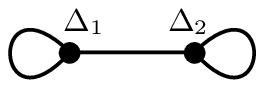}\qquad}
    \caption{Examples of triangulations}
\end{figure}

\begin{example} \label{ex-tri}
    Figure~\ref{fig-kb} illustrates a 2-manifold triangulation with
    $n=2$ simplices whose
    underlying topological space is a Klein bottle.  As indicated by the
    arrowheads, we identify the following pairs of facets (i.e., edges):
    \[
    \Delta_1\of02 \longleftrightarrow \Delta_2\of20,\quad
    \Delta_1\of01 \longleftrightarrow \Delta_1\of12,\quad
    \Delta_2\of01 \longleftrightarrow \Delta_2\of12. \]
    The resulting triangulation has one vertex (since
    all three vertices of $\Delta_1$ and all three vertices of
    $\Delta_2$ become identified together),
    and three edges (labelled $e,f,g$ in the diagram).

    If we label each edge so that vertices $0$ and $1$ are at the
    base and the tip of the arrow respectively, then vertex $0$ of edge
    $e$ corresponds to the individual triangle vertices $\Delta_1\of0$
    and $\Delta_2\of2$, and vertex $1$ of edge $e$ corresponds to
    $\Delta_1\of2$ and $\Delta_2\of0$.

    Figure~\ref{fig-lst} illustrates a 3-manifold triangulation
    with just $n=1$ simplex whose underlying topological space is the
    solid torus $B^2 \times S^1$ \cite{jaco03-0-efficiency}.
    Here we identify facets $\Delta_1\of012 \longleftrightarrow \Delta_1\of123$,
    and leave the other two facets as boundary.
    The resulting triangulation has just
    one vertex, three edges, and three 2-faces.
\end{example}

Let $\tri$ be a $d$-dimensional triangulation.
The \emph{size} of $\tri$, denoted $|\tri|$, is the number of
simplices (i.e., $d$-faces) in $\tri$.
Note that the total number of faces
of \emph{any} dimension is at most $2^{d+1}|\tri|$, and is hence
linear in $|\tri|$ for fixed dimension $d$.

The \emph{dual graph} of $\tri$, denoted $\dual(\tri)$,
is the multigraph whose nodes correspond to simplices
and whose arcs correspond to identified pairs of facets.
In particular, $\dual(\tri)$ has precisely $|\tri|$ nodes,
and each node has degree at most $d+1$.
Loops may occur in $\dual(\tri)$ if two facets of the same simplex are
identified, and parallel arcs may occur if different
facets of some simplex $\Delta_i$ are identified with (different)
facets of the same
simplex $\Delta_j$.  Figure~\ref{fig-dual} illustrates the dual graph of
the Klein bottle from Example~\ref{ex-tri}.

%
%

\section{Edge-coloured graphs} \label{s-coloured}

Here we prove that the three variants of Courcelle's theorem on simple
graphs (Corollary~\ref{c-courcelle-test},
Theorem~\ref{t-courcelle-opt} and Theorem~\ref{t-courcelle-eval})
also hold for edge-coloured graphs with a fixed number of colours.
Although the results here are unsurprising and draw on standard
techniques, they are necessary as a stepping stone for Section~\ref{s-meta},
where we use edge-coloured graphs to encode
the full structure of a triangulation.


\subsection{MSO logic on edge-coloured graphs}

For any fixed number of colours $k \in \N$, we can extend
MSO logic to the setting of edge-coloured graphs as follows.
We allow all of the constructs of MSO logic on simple graphs, as
outlined in Section~\ref{s-msol}.  In additional, we support:
\begin{itemize}
    \item $k$ unary colour relations
    $\col_1(e), \ldots, \col_k(e)$,
    where $\col_i(e)$ encodes the fact that $e$ is an arc of the
    $i$th colour;
    \item $k$ binary adjacency relations
    $\adj_1(v,v'), \ldots, \adj_k(v,v')$,
    where $\adj_i(v,v')$ encodes the fact that $v$ and $v'$ are the two
    endpoints of some common arc of the $i$th colour.
\end{itemize}
Note that the relations $\adj_i$ are for convenience only, and can
easily be encoded using the other constructs available.

As usual, if $G=(V,E,C)$ is an edge-coloured graph with $|C|=k$ colours,
and $\phi$ is an MSO sentence using the additional constructs outlined
above, the notation $G \models \phi$ indicates that the interpretation
of $\phi$ in the graph $G$ is a true statement.

We define extremum and evaluation problems exactly as before:
a restricted MSO extremum problem minimises a rational function over
sizes of set variables as in Section~\ref{s-extremum}, and
an MSO evaluation problem computes an additive or multiplicative quantity
using weights on the nodes and/or arcs as in Section~\ref{s-evaluation}.


\subsection{Metatheorems on edge-coloured graphs}
\label{s-col-meta}

Here we translate Courcelle's theorem and its variants to edge-coloured
graphs.  The basic idea is, for any edge-coloured graph $G$,
to construct an associated simple graph $\overline{G}$ for which:
\begin{itemize}
\item the size $|\overline{G}|$ is linear in $|G|$;
\item the treewidth $\tw(\overline{G})$ is at worst linear in $\tw(G)$; and
\item MSO formulae on $G$ translate to MSO formulae on
$\overline{G}$.
\end{itemize}
The coloured variants of Courcelle's theorem then fall out naturally,
as seen in Theorem~\ref{t-courcelle-coloured}.
The details follow.

\begin{construction} \label{const-simple}
    Let $G=(V,E,C)$ be any edge-coloured graph, with colours
    $C=\{c_1,\ldots,c_k\}$.
    We construct the associated simple graph $\overline{G}$ as follows:
    \begin{itemize}
        \item For each colour $c_i$ we add $i+2$ nodes
        $\kappa_{i,1},\ldots,\kappa_{i,i+2}$ to $\overline{G}$
        plus all $\binom{i+2}{2}$ possible arcs between them.
        In other words, we insert $k$ cliques of sizes $3,\ldots,k+2$.

        \item For each node $v \in V$ we add a corresponding node
        $\overline{v}$ to $\overline{G}$.

        \item For each arc $e = \{(v,w),c) \in E$ we add
        a corresponding node $\overline{e}$ to $\overline{G}$,
        plus three arcs $\{\overline{e},\overline{v}\}$,
        $\{\overline{e},\overline{w}\}$, and
        $\{\overline{e},\kappa_{c,1}\}$.
    \end{itemize}
\end{construction}

\begin{figure}[tb]
    \subfigure[The edge-coloured graph $G$\label{fig-simple-graph}]{%
        \qquad\includegraphics[scale=0.9]{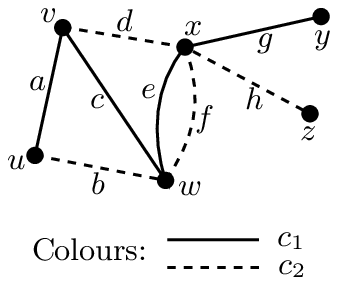}\qquad}
    \qquad
    \subfigure[The associated simple graph $\overline{G}$\label{fig-simple-plain}]{%
        \includegraphics[scale=0.9]{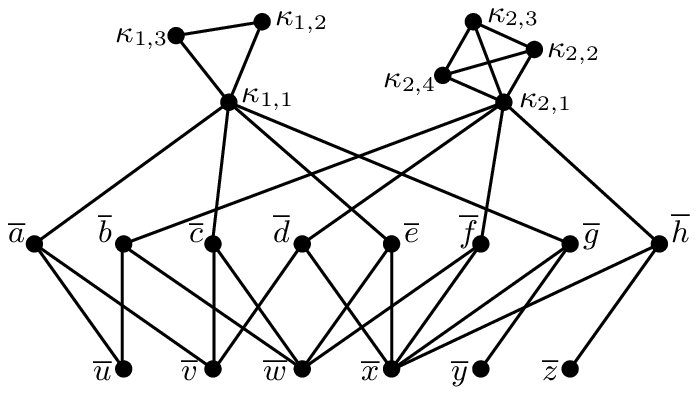}}
    \caption{Converting an edge-coloured graph into a simple graph}
    \label{fig-simple}
\end{figure}

Figure~\ref{fig-simple} illustrates this construction.
It is immediate from the definition of an edge-coloured graph (see
Section~\ref{s-tw}) that $\overline{G}$ is indeed a simple graph as claimed.

\begin{lemma} \label{l-simple-size}
    For any edge-coloured graph $G=(V,E,C)$ with $|C|=k$ colours, we have
    $|\overline{G}| = |V|+4|E|+\binom{k+4}{3} - 4$.
\end{lemma}

\begin{proof}
    It is clear by construction
    that the number of nodes of $\overline{G}$ is
    $|V|+|E|+\sum_{i=3}^{k+2} i =
    |V|+|E|+ \binom{k+3}{2}-3$,
    and the number of arcs of $\overline{G}$ is
    $3|E|+\sum_{i=3}^{k+2} \binom{i}{2} =
    3|E|+\binom{k+3}{3} - 1$.
    The result then follows from the identity
    $\binom{k+3}{2}+\binom{k+3}{3} = \binom{k+4}{3}$.
\end{proof}

\begin{lemma} \label{l-simple-tw}
    For any edge-coloured graph $G=(V,E,C)$ with $|C|=k$ colours, we have
    $\tw(\overline{G}) \leq \tw(G) + \binom{k+3}{2} - 1$.
\end{lemma}

\begin{figure}[tb]
    \subfigure[A tree decomposition of $G$\label{fig-simple-tw-graph}]{%
        \qquad\includegraphics[scale=0.75]{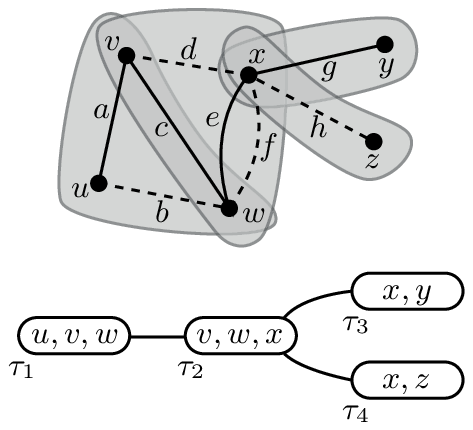}\qquad}
    \subfigure[The corresponding tree decomposition of $\overline{G}$%
        \label{fig-simple-tw-plain}]{%
        \includegraphics[scale=0.75]{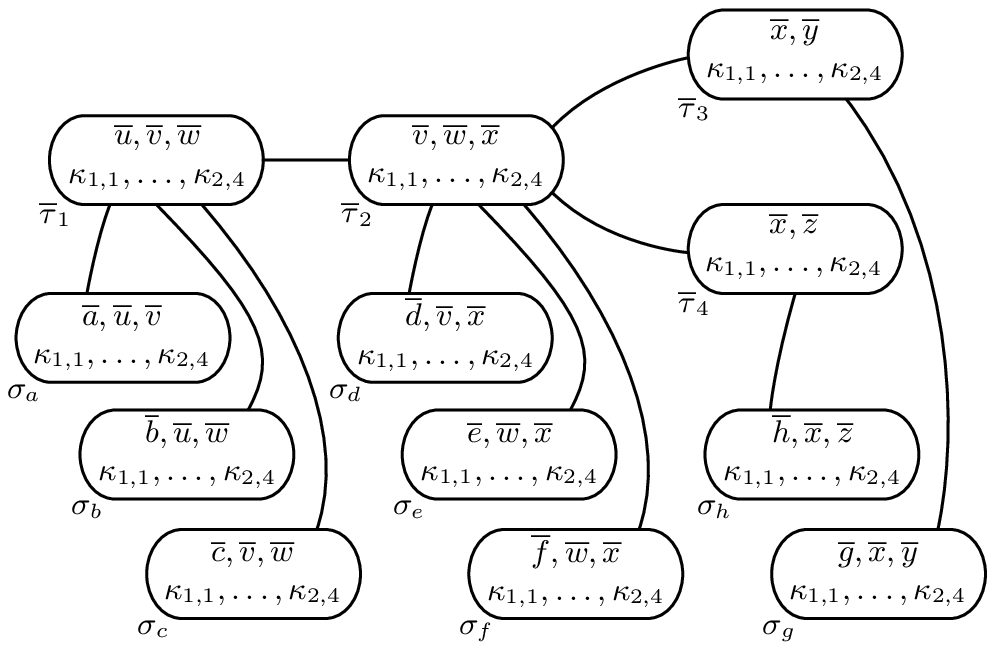}}
    \caption{Bounding the treewidth of the simple graph $\overline{G}$}
    \label{fig-simple-tw}
\end{figure}

\begin{proof}
    Suppose we have a tree decomposition of $G$ with
    underlying tree $T$.  From this we build a tree decomposition of
    $\overline{G}$ with underlying tree $\overline{T}$ as follows
    (see Figure~\ref{fig-simple-tw} for an illustration):
    \begin{enumerate}
        \item
        \label{en-tree-copy}
        In the original tree decomposition, let
        $\tau$ be a node of $T$ and let $B_\tau = \{v_1,\ldots,v_m\}$
        be the corresponding bag.
        In the new tree decomposition, we create a corresponding node
        $\overline{\tau}$ of $\overline{T}$, whose bag
        $B_{\overline{\tau}}$ contains
        the corresponding nodes
        $\overline{v}_1,\ldots,\overline{v}_m$
        plus all of the clique nodes $\kappa_{i,j}$.

        \item
        \label{en-tree-arc}
        For each arc $e=\{(v,w),c\}$ of $G$ we create a new node
        $\sigma_e$ of $\overline{T}$,
        whose corresponding bag $B_{\sigma_e}$ contains
        $\overline{e}$, $\overline{v}$ and $\overline{w}$
        plus all of the clique nodes $\kappa_{i,j}$.

        \item
        \label{en-tree-connect}
        We connect the nodes of $\overline{T}$ as follows.
        For each pair of adjacent nodes $\tau_1,\tau_2$ of $T$,
        we connect the corresponding nodes
        $\overline{\tau}_1,\overline{\tau}_2$ in $\overline{T}$.
        For each arc $e=\{(v,w),c\}$ of $G$, we connect the
        corresponding node $\sigma_e$ of $\overline{T}$
        to one (and only one) arbitrarily chosen
        node $\overline{\tau}$ for which
        $\overline{v},\overline{w} \in B_{\overline{\tau}}$.
    \end{enumerate}

    Note that the choice in step~(\ref{en-tree-connect})
    is always possible: since we began with a tree decomposition of $G$,
    some node $\tau$ of $T$ must have $v,w \in B_\tau$
    and thus $\overline{v},\overline{w} \in B_{\overline{\tau}}$.

    It is simple to show that this construction does indeed yield a tree
    decomposition of $\overline{G}$: all necessary properties of a
    tree decomposition follow directly from the construction above and
    the fact that we began with a tree decomposition of $G$.

    It remains to compute the bag sizes in our new tree decomposition.
    For each bag $B_{\overline{\tau}}$ created in
    step~(\ref{en-tree-copy}), we have
    $|B_{\overline{\tau}}| = |B_\tau| + \sum_{i=3}^{k+2} i =
    |B_\tau| + \binom{k+3}{2} - 3$.
    For each bag $B_{\sigma_e}$ created in step~(\ref{en-tree-arc}),
    we have
    $|B_{\sigma_e}| = 3 + \sum_{i=3}^{k+2} i = \binom{k+3}{2}$.
    Therefore
    $\tw(\overline{G}) \leq \max\left\{\tw(G)+\binom{k+3}{2} - 3,
        \binom{k+3}{2}-1\right\} \leq \tw(G)+\binom{k+3}{2}-1$.
\end{proof}

\begin{lemma} \label{l-simple-mso}
    For any fixed $k \in \N$,
    every MSO formula $\phi(x_1,\ldots,x_t)$
    on edge-coloured graphs with $k$ colours
    has a corresponding MSO formula
    $\overline{\phi}(\overline{x}_1,\ldots,\overline{x}_t)$ on simple graphs
    for which,
    for any edge-coloured graph $G(V,E,C)$ with $|C|=k$ colours:
    \begin{itemize}
        \item Any assignment of values to $x_1,\ldots,x_t$ for which
        $G \models \phi(x_1,\ldots,x_t)$ yields a corresponding
        assignment of values to $\overline{x}_1,\ldots,\overline{x}_t$
        for which
        $\overline{G} \models
        \overline{\phi}(\overline{x}_1,\ldots,\overline{x}_t)$,
        and this assignment is obtained as follows:
        \begin{itemize}
            \item if $x_i = v$ for some node $v \in V$, then
            $\overline{x}_i = \overline{v}$;
            \item if $x_i = e$ for some arc $e \in E$, then
            $\overline{x}_i = \overline{e}$
            (note that $\overline{x}_i$ is a node of
            $\overline{G}$, not an arc);
            \item if $x_i$ is a set $\{v_i\} \subseteq V$ or
            $\{e_i\} \subseteq E$, then $\overline{x}_i$ is the corresponding
            set $\{\overline{v}_i\}$ or $\{\overline{e}_i\}$.
        \end{itemize}
        \item Conversely, any assignment of values to
        $\overline{x}_1,\ldots,\overline{x}_t$ for which
        $\overline{G} \models
        \overline{\phi}(\overline{x}_1,\ldots,\overline{x}_t)$
        is derived from an
        assignment of values to $x_1,\ldots,x_t$ for which
        $G \models \phi(x_1,\ldots,x_t)$, as described above.
    \end{itemize}
\end{lemma}

In particular, this means that if $\phi$ is a sentence then
$G \models \phi$ if and only if $\overline{G} \models \overline{\phi}$,
and if $\phi$ has free variables then the solutions
to $G \models \phi(x_1,\ldots,x_t)$ are in
bijection with the solutions to
$\overline{G} \models \overline{\phi}(\overline{x}_1,\ldots,\overline{x}_t)$.

\begin{proof}
    The proof works in three stages:
    (i)~we develop additional ``helper'' relations to constrain the roles
    that variables can play in $\overline{\phi}$;
    (ii)~we translate each piece of the formula $\phi$
    so that statements about $G$ in $\phi$
    become statements about $\overline{G}$ in $\overline{\phi}$, and
    thus solutions to $G \models \phi(x_1,\ldots,x_t)$ become
    solutions to $\overline{G} \models
    \overline{\phi}(\overline{x}_1,\ldots,\overline{x}_t)$; and then
    (iii)~we add additional constraints to $\overline{\phi}$ to avoid any
    additional (and unwanted) solutions to
    $\overline{G} \models
    \overline{\phi}(\overline{x}_1,\ldots,\overline{x}_t)$.

    In the first stage, we develop the following ``helper'' unary relations
    using MSO logic on simple graphs.  These relations only make sense
    when interpreting $\overline{\phi}$ on a simple graph $\overline{G}$
    that was built using Construction~\ref{const-simple}; in other
    contexts they are well-defined but meaningless.
    \begin{itemize}
        \item $\isnode(x)$ indicates that a variable in
        $\overline{\phi}$ represents a node of the original
        edge-coloured graph $G$; that is,
        $x=\overline{v}$ for some $v \in V$.

        \item $\isarc(x)$ indicates that a variable in
        $\overline{\phi}$ represents an arc of the original
        edge-coloured graph $G$; that is,
        $x=\overline{e}$ for some arc $e \in E$.

        \item For each $i=1,\ldots,k$, $\iscol_i(x)$ indicates that a
        variable in $\overline{\phi}$ represents one of the nodes of the
        clique representing the $i$th colour; that is,
        $x=\kappa_{i,j}$ for some $j$.
    \end{itemize}

    These relations are simple (but messy) to piece together using MSO logic.
    For each fixed $i$, the relation $\iscol_i(x)$ is true if and only if
    $x$ represents a node that belongs to an $i$-clique but not an
    $(i+1)$-clique.  The relation $\isarc(x)$ is true if and only if $x$
    is a node that does not belong to a triangle, but is adjacent to a
    node that does.  The relation $\isnode(x)$ is true if and only if
    $x$ is a node for which none of $\isarc(x)$ or $\iscol_i(x)$ are true.

    In the second stage, we translate each piece of $\phi$ to
    a piece of $\overline{\phi}$ as follows:
    \begin{itemize}
        \item Each variable $x_i$ in $\phi$ is replaced by
        $\overline{x}_i$ in $\overline{\phi}$.

        \item Standard logical operations (such as $\wedge$, $\vee$,
        $\neg$, $\rightarrow$ and so on), standard quantifiers
        ($\forall$ and $\exists$), and the equality and inclusion
        relations ($=$ and $\in$) are copied directly from $\phi$ to
        $\overline{\phi}$.

        \item The incidence relation $\inc(e,v)$ in $\phi$ is replaced
        by the phrase
        $\isarc(e) \wedge \isnode(v) \wedge \adj(e,v)$
        in $\overline{\phi}$.

        \item Each colour relation
        $\col_i(e)$ in $\phi$ is replaced by the phrase
        $\exists v \ \isarc(e) \wedge \iscol_i(v) \wedge \adj(e,v)$
        in $\overline{\phi}$.

        \item The adjacency relation
        $\adj(v,v')$ in $\phi$ is replaced by the phrase
        $\exists e \ \isnode(v) \wedge \isnode(v') \wedge
        \adj(v,e) \wedge \adj(v',e) \wedge v \neq v'$
        in $\overline{\phi}$.

        \item Each coloured adjacency relation
        $\adj_i(v,v')$ in $\phi$ can be replaced with an equivalent statement
        in $\phi$ using $\inc(\cdot,\cdot)$ and $\col_i(\cdot)$, and then
        translated to $\overline{\phi}$ as described above.
    \end{itemize}

    It follows directly from this translation process that,
    if an assignment of values to $x_1,\ldots,x_t$ gives
    $G \models \phi(x_1,\ldots,x_t)$, then the corresponding
    assignment of values to $\overline{x}_1,\ldots,\overline{x}_t$
    as described in the lemma statement gives
    $\overline{G} \models
    \overline{\phi}(\overline{x}_1,\ldots,\overline{x}_t)$.

    In the third stage, we eliminate any additional and unwanted solutions to
    $\overline{G} \models
    \overline{\phi}(\overline{x}_1,\ldots,\overline{x}_t)$,
    by insisting that for every variable $x$ in $\phi$,
    the translated variable $\overline{x}$ in $\overline{\phi}$
    must still represent
    a node, arc, set of nodes or set of arcs in the source graph $G$.
    This ensures that logical statements about such variables
    $\overline{x}_i$ in
    $\overline{\phi}$ translate correctly back to logical statements
    about variables $x_i$ in $\phi$.
    Specifically, for each variable $x$ in $\phi$
    (either bound or free), we add the corresponding clause to
    $\overline{\phi}$:
    \[ \isnode(\overline{x}) \vee
       \isarc(\overline{x}) \vee
       [\forall z\ z \in \overline{x} \rightarrow \isnode(z)] \vee
       [\forall z\ z \in \overline{x} \rightarrow \isarc(z)].
       \qedhere\]
\end{proof}

\begin{theorem} \label{t-courcelle-coloured}
    Let $K$ be any class of edge-coloured graphs with fixed colour set
    $C=\{c_1,\ldots,c_k\}$ and with universally bounded treewidth.
    Then:
    \begin{itemize}
        \item For any fixed MSO sentence $\phi$, it is possible to test
        whether $G \models \phi$ for edge-coloured graphs $G \in K$
        in time $O(|G|)$;
        \item For any restricted MSO extremum problem $P$, it is
        possible to solve $P$ for edge-coloured graphs $G \in K$ in time
        $O(|G|)$ under the uniform cost measure;
        \item For any MSO evaluation problem $P$, it is possible to
        solve $P$ for edge-coloured graphs $G \in K$ in time $O(|G|)$
        under the uniform cost measure.
    \end{itemize}
\end{theorem}

\begin{proof}
    Let $P$ be any such problem, and let $\phi$ be its underlying
    MSO sentence or formula.  We first use Lemma~\ref{l-simple-mso}
    to translate $\phi$ to $\overline{\phi}$,
    yielding a new problem $\overline{P}$ in the setting of simple graphs.
    Then, for any edge-coloured graph $G$ given as input to $P$,
    we construct the simple graph $\overline{G}$ and solve
    $\overline{P}$ for $\overline{G}$ instead.
    Lemma~\ref{l-simple-mso} ensures that the solutions to both problems
    are the same.

    For extremum problems, we note that the sizes of the sets
    in each solution to $\phi$ are equal to the sizes of the sets in
    the corresponding solution to $\overline{\phi}$
    (i.e., the value of the extremum
    does not change).  For evaluation problems, the input weight on
    a node $v$ or arc $e$ of $G$ becomes the same input weight
    on the corresponding node $\overline{v}$ or $\overline{e}$ of
    $\overline{G}$; all remaining nodes and arcs of $\overline{G}$
    are assigned trivial weights ($0$ or $1$ according to whether the
    evaluation problem is additive or multiplicative).

    It is clear from Construction~\ref{const-simple} that we can build
    the simple graph $\overline{G}$ in time $O(|G|)$,
    and Lemmata~\ref{l-simple-size} and \ref{l-simple-tw}
    ensure that $\overline{G}$
    has universally bounded treewidth and $O(|G|)$ size.
    The result now follows directly from the
    three original variants of Courcelle's theorem
    (Corollary~\ref{c-courcelle-test},
    Theorem~\ref{t-courcelle-opt} and Theorem~\ref{t-courcelle-eval}).
\end{proof}

%
%

\section{Triangulations} \label{s-meta}

In this section we prove our main result, that all three variants of
Courcelle's theorem hold for triangulations of fixed dimension
(Theorem~\ref{t-tri}).


\subsection{MSO logic on triangulations}

Our first task is to extend MSO logic to the setting of
$d$-dimensional triangulations, for fixed dimension $d \in \N$.
Here we define MSO logic to support:
\begin{itemize}
    \item all of the standard boolean operations of propositional logic
    ($\wedge$, $\vee$, $\neg$, $\rightarrow$, and so on);
    \item for each $i=0,\ldots,d$, variables to represent
    $i$-faces of a triangulation, or sets of $i$-faces of a triangulation;
    \item the standard quantifiers ($\forall$ and $\exists$) and
    the binary equality and inclusion relations ($=$ and $\in$),
    which may be applied to any of these variable types;
    \item for each $i=0,\ldots,d-1$ and for each ordered sequence
    $\pi_0,\ldots,\pi_i$ of distinct integers from $\{0,\ldots,d\}$,
    a subface relation $\leq_{\pi_0 \ldots \pi_i}$.
\end{itemize}

The relation $(f \leq_{\pi_0 \ldots \pi_i} s)$ indicates that
$f$ is an $i$-face of the triangulation,
$s$ is a simplex of the triangulation,
and that $f$ is identified with the subface of $s$ formed by the simplex
vertices $\pi_0,\ldots,\pi_i$, in a way that vertices $0,\ldots,i$ of
the face $f$ correspond to vertices $\pi_0,\ldots,\pi_i$ of the simplex $s$.

\begin{example}
    Recall the Klein bottle example illustrated in Figure~\ref{fig-kb}.
    Here the three edges $e,f,g$ satisfy the subface relations
    \[ \begin{array}{l@{\qquad}l@{\qquad}l}
        e \leq_{02} \Delta_1 & f \leq_{01} \Delta_1 & g \leq_{01} \Delta_2 \\
        e \leq_{20} \Delta_2 & f \leq_{12} \Delta_1 & g \leq_{12} \Delta_2.
    \end{array} \]
    The triangulation has only one vertex (since all vertices of
    $\Delta_1$ and $\Delta_2$ are identified together); call this
    vertex $v$.  Then $v$ satisfies all possible subface relations
    \[ \begin{array}{l@{\qquad}l@{\qquad}l}
        v \leq_{0} \Delta_1 & v \leq_{1} \Delta_1 & v \leq_{2} \Delta_1 \\
        v \leq_{0} \Delta_2 & v \leq_{1} \Delta_2 & v \leq_{2} \Delta_2.
    \end{array} \]
\end{example}

\begin{notation}
In general, we will adopt the convention that lower-case letters
$s,t,\ldots$ represent simplices and
$f^{(i)},g^{(i)},\ldots$ represent $i$-faces,
and that upper-case letters
$S,T,\ldots$ and $F^{(i)},G^{(i)},\ldots$ represent sets of simplices and
sets of $i$-faces respectively.
\end{notation}

\begin{example}[Orientability]
    Consider $d=2$, i.e., the case of triangulated surfaces.
    We will construct an MSO sentence $\phi$ stating that a
    triangulation represents an \emph{orientable} surface.

    Recall that a 2-dimensional triangulation is orientable if and only if
    each triangle can be assigned an orientation (clockwise or
    anticlockwise) so that adjacent triangles have compatible
    orientations, as illustrated in Figure~\ref{fig-orientability}.

    \begin{figure}[tb]
        \centering
        \includegraphics[scale=1]{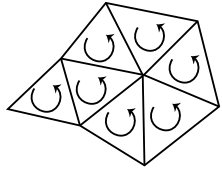}
        \caption{Adjacent triangles have compatible orientations}
        \label{fig-orientability}
    \end{figure}

    The sentence $\phi$ is given below.
    Here $S_+$ and $S_-$ are variables denoting sets of simplices
    with clockwise and anticlockwise orientations respectively.
    To encode the compatibility constraint on adjacent triangles,
    we introduce variables $u$ and $v$ to represent two adjacent
    simplices, and $f^{(1)}$ to represent the common edge along which they are
    joined.
    \[ \begin{array}{l@{\;\rightarrow\;}l}
        \multicolumn{2}{l}{
            \exists S_+ \exists S_-
            \forall s\ \forall f^{(1)}\ \forall u \ \forall v} \\
        \multicolumn{2}{l}{
            (s \in S_+ \vee s \in S_-) \wedge
            \neg (s \in S_+ \wedge s \in S_-) \wedge {}} \\
        \left[(f^{(1)} \leq_{01} u \wedge f^{(1)} \leq_{01} v \wedge
            u \neq v)\right.
            &\left.((u \in S_+ \wedge v \in S_-) \vee
              (u \in S_- \wedge v \in S_+))\right] \wedge {}\\
        \left[(f^{(1)} \leq_{01} u \wedge f^{(1)} \leq_{10} v) \right.
            &\left.((u \in S_+ \wedge v \in S_+) \vee
              (u \in S_- \wedge v \in S_-))\right] \wedge {}\\
        \left[(f^{(1)} \leq_{01} u \wedge f^{(1)} \leq_{02} v) \right.
            &\left.((u \in S_+ \wedge v \in S_+) \vee
              (u \in S_- \wedge v \in S_-))\right] \wedge {}\\
        \left[(f^{(1)} \leq_{01} u \wedge f^{(1)} \leq_{20} v) \right.
            &\left.((u \in S_+ \wedge v \in S_-) \vee
              (u \in S_- \wedge v \in S_+))\right] \wedge {}\\
        \multicolumn{2}{c}{\vdots} \\
        \left[(f^{(1)} \leq_{21} u \wedge f^{(1)} \leq_{21} v \wedge
            u \neq v)\right.
            &\left.((u \in S_+ \wedge v \in S_-) \vee
              (u \in S_- \wedge v \in S_+))\right].
       \end{array} \]

    The first line of this sentence is just quantifiers, and
    the second line ensures that $S_+$ and $S_-$ partition the simplices.
    The remaining lines encode the fact that adjacent simplices must have
    compatible orientations.  They do this by iterating through all
    $6 \times 6$ possible
    ways in which $f^{(1)}$ could appear as both an edge of $u$ and an edge
    of $v$ (making $u$ and $v$ adjacent), and in each case a simple
    parity check determines whether $u$ and $v$ must have the same orientation
    (as in the case $f^{(1)} \leq_{01} u \wedge f^{(1)} \leq_{02} v$)
    or opposite orientations
    (as in the case $f^{(1)} \leq_{01} u \wedge f^{(1)} \leq_{20} v$).

    For the six cases where $u$ and $v$ use the \emph{same} subface
    relation (e.g., the first and last cases above), we add the
    additional requirement
    $u \neq v$ to ensure that triangles $u$ and $v$
    lie on opposite sides of the edge $f^{(1)}$.
\end{example}

We return now to notation and definitions.
If $\tri$ is a $d$-dimensional triangulation and $\phi$ is an
MSO sentence as outlined above, the notation $\tri \models \phi$
indicates (as usual)
that the interpretation of $\phi$ in the triangulation $\tri$
is a true statement.

We define extremum and evaluation problems as before, though this time we
must alter the latter definition so that weights are given to
the faces of a triangulation (instead of nodes and arcs of a graph).
Specifically, in the setting of MSO logic on $d$-dimensional triangulations:
\begin{itemize}
\item
A \emph{restricted MSO extremum problem}
consists of an MSO formula $\phi(A_1,\ldots,A_t)$
with free set variables $A_1,\ldots,A_t$,
and a rational linear function $g(x_1,\ldots,x_t)$.  Its interpretation
is as follows: given a $d$-dimensional triangulation $\tri$ as input,
we are asked to minimise $g(|A_1|,\ldots,|A_t|)$ over all sets
$A_1,\ldots,A_t$ for which $\tri \models \phi(A_1,\ldots,A_t)$.
\item
An \emph{MSO evaluation problem}
consists of an MSO formula $\phi(A_1,\ldots,A_t)$
with $t$ free set variables $A_1,\ldots,A_t$.
The input to the problem is a $d$-dimensional triangulation $\tri$,
together with $t$ weight functions
$w_1,\ldots,w_t \co F_0 \sqcup \ldots \sqcup F_d \to R$,
where $F_i$ denotes the set of all $i$-faces of $\tri$,
and $R$ is some ring or field.
The problem then asks us to compute one of the quantities
\[ \sum_{\tri \models \phi(A_1,\ldots,A_t)}
    \left\{
    \sum_{i=1}^t \sum_{x_i \in A_i} w_i(x_i)
    \right\}
    \quad\mbox{or}\quad
   \sum_{\tri \models \phi(A_1,\ldots,A_t)}
    \left\{
    \prod_{i=1}^t \prod_{x_i \in A_i} w_i(x_i)
    \right\}. \]
\end{itemize}

We note again that MSO evaluation problems include
counting problems as a special case (simply use the multiplicative variant
with all weights set to $1$).
For examples of extremum and evaluation problems,
see discrete Morse matchings
(Section~\ref{s-morse}) and the Turaev-Viro invariants
(Section~\ref{s-tv}) respectively.


\subsection{Metatheorems on triangulations}

We begin this section by introducing coloured Hasse diagrams, which allow
us to translate between triangulations and edge-coloured graphs.
Following this we present and prove the three variants of Courcelle's
theorem on triangulations (Theorem~\ref{t-tri}).

In settings such as simplicial complexes or polytopes, a Hasse
diagram is a graph with a node for each face, and an
arc whenever an $i$-face appears as a subface of an $(i+1)$-face.
Here we extend Hasse diagrams to include more precise information about
\emph{how} subfaces are embedded, and to support situations
in which a face $f$ is embedded as a subface of another face $f'$ more
than once.

\begin{defn} \label{d-hasse}
    Let $\tri$ be a $d$-dimensional triangulation.
    Then the \emph{coloured Hasse diagram} of $\tri$,
    denoted $\hasse(\tri)$, is the following edge-coloured graph:
    \begin{itemize}
        \item The colours of $\hasse(\tri)$ are,
        for all $i=0,\ldots,d-1$,
        all ordered sequences $\pi_0 \ldots \pi_i$ of
        distinct integers from the set $\{0,\ldots,i+1\}$
        (so exactly one integer from the set is not used).
        We also allow an additional ``empty colour'', denoted by a dash ($-$).
        This gives $\sum_{i=1}^{d+1} i!$ colours in total.

        For example, for $d=2$ we use the following colours:
        $-,0,1,01,02,10,12,\allowbreak 20,21$.  Note that there is no colour
        labelled $2$ in this list.

        \item The nodes of $\hasse(\tri)$ are the $i$-faces of $\tri$,
        for all $i=0,\ldots,d$.
        We also add an extra node for the ``empty face'',
        denoted $\emptyset$.

        \item The arcs of $\hasse(\tri)$ are as follows.
        If $f$ is an $i$-face and $g$ is an $(i+1)$-face with
        $f \leq_{\pi_0\ldots\pi_i} g$, then we place an arc between
        nodes $f$ and $g$ of colour $\pi_0\ldots\pi_i$.
        We also add an arc between each vertex of the triangulation
        and the empty node $\emptyset$, coloured by the
        ``empty colour'' $-$.
    \end{itemize}
\end{defn}

\begin{figure}[tb]
    \centering
    \subfigure[A Klein bottle $\mathcal{K}$]{%
        \includegraphics[scale=0.9]{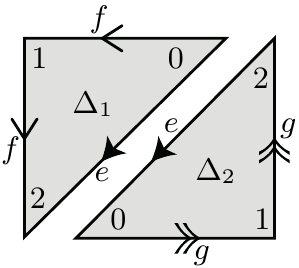}}
    \qquad
    \subfigure[The coloured Hasse diagram $\hasse(\mathcal{K})$]{%
        \qquad\includegraphics[scale=0.9]{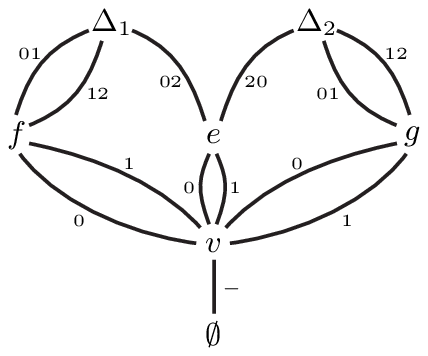}\qquad}
    \caption{An example of a coloured Hasse diagram}
    \label{fig-hasse}
\end{figure}

\begin{example}
    Figure~\ref{fig-hasse} shows a $2$-dimensional Klein bottle
    (the same as seen earlier in Example~\ref{ex-tri}), alongside its
    coloured Hasse diagram.  For consistency with the
    earlier example, we label the common vertex as $v$,
    the three edges as $e,f,g$, and the two
    triangles as $\Delta_1,\Delta_2$.
\end{example}

We now establish a series of results that allow us to convert problems
about triangulations into problems about edge-coloured graphs.
Recall from Section~\ref{s-prelim-tri} that if $\tri$ is a triangulation
then $|\tri|$ denotes the number of
top-dimensional simplices of $\tri$, and $\dual(\tri)$ denotes the
dual graph of $\tri$.
Following the pattern established in Section~\ref{s-col-meta},
the following lemmata essentially show
that for fixed dimension $d$:
\begin{itemize}
\item the size $|\hasse(\tri)|$ is linear in $|\tri|$;
\item the treewidth $\tw(\hasse(\tri))$ is at worst linear in
$\tw(\dual(\tri))$; and
\item MSO formulae on $\tri$ translate to MSO formulae on $\hasse(\tri)$.
\end{itemize}
The triangulation-based variants of Courcelle's theorem then follow
naturally from these results, as seen in Theorem~\ref{t-tri}.

\begin{lemma} \label{l-hasse-size}
    For any $d$-dimensional triangulation $\tri$ of positive size, we have
    $|\hasse(\tri)| \leq 2^d (d+3) \cdot |\tri|$.
\end{lemma}

\begin{proof}
    Each individual $d$-simplex has $\binom{d+1}{i+1}$ distinct $i$-faces
    for each $i=0,\ldots,d$,
    and so the triangulation has
    $\leq \binom{d+1}{i+1} |\tri|$ distinct $i$-faces in total
    (typically there are fewer, since
    individual faces of simplices are identified together in the
    triangulation).
    Therefore $\hasse(\tri)$ has at most
    $1 + \sum_{i=0}^d \binom{d+1}{i+1}|\tri|
    \leq 2^{d+1}|\tri|$ nodes, including the special node $\emptyset$.

    For each $i=1,\ldots,d$,
    each $i$-face $g$ of the triangulation has precisely
    $(i+1)$ subface relationships of the form $f \leq_{\pi_0\ldots\pi_{i-1}} g$
    where $f$ is an $(i-1)$-face.
    In addition, each $0$-face (vertex) of the triangulation has
    exactly one arc running to the empty node $\emptyset$.
    Therefore $\hasse(\tri)$ has at most
    $\binom{d+1}{1}|\tri| + \sum_{i=1}^{d} (i+1) \binom{d+1}{i+1}|\tri|
    = 2^d(d+1) |\tri|$ arcs.

    Combining these counts gives
    $|\hasse(\tri)| \leq 2^{d+1}|\tri| + 2^d(d+1) |\tri| =
    2^d(d+3) |\tri|$.
\end{proof}

\begin{lemma} \label{l-hasse-tw}
    For any $d$-dimensional triangulation $\tri$ of positive size, we have
    $\tw(\hasse(\tri)) \leq (2^{d+1}-1)(\tw(\dual(\tri)) + 1)$.
\end{lemma}

\begin{proof}
    Suppose we have a tree decomposition of $\dual(\tri)$, with
    underlying tree $T$ and bags $\{B_\tau\}$.
    From this we build a tree decomposition of $\hasse(\tri)$
    using the same tree $T$ but with different bags $\{B'_\tau\}$.

    Specifically, let $\tau$ be a node of $T$.
    The bag $B_\tau$ contains a set of nodes of
    $\dual(\tri)$, or equivalently a set of simplices
    $\Delta_1,\ldots,\Delta_t$ of $\tri$.
    We define the corresponding bag $B'_\tau$ to contain those nodes of
    $\hasse(\tri)$ that denote all of $\Delta_1,\ldots,\Delta_t$,
    all of their subfaces (of any dimension),
    and also the empty face $\emptyset$.

    Assume for now that this is indeed a tree decomposition of $\hasse(\tri)$
    (we prove this shortly).
    Each $\Delta_i$ has $\leq 2^{d+1}-2$ non-empty proper subfaces
    (possibly fewer if some of these subfaces are identified),
    and so $|B'_\tau| \leq (2^{d+1}-1)|B_\tau|+1$.  Therefore
    $\tw(\hasse(\tri))+1 \leq (2^{d+1}-1)[\tw(\dual(\tri))+1]+1$,
    which produces the final result.

    It remains to show that our construction does yield a tree
    decomposition of $\hasse(\tri)$.
    It is clear that every node of $\hasse(\tri)$ belongs to a bag
    $B'_\tau$.
    Moreover, every arc of $\hasse(\tri)$ has both endpoints in some
    common bag $B'_\tau$, since each bag containing a face $f$ will
    also contain every subface of $f$.

    To prove the subtree connectivity property for the bags $\{B'_\tau\}$:
    \begin{itemize}
        \item Every bag contains $\emptyset$.
        \item For each simplex $\Delta_i$, we already know from the
        original tree decomposition that the bags
        containing $\Delta_i$ form a connected subtree of $T$.
        \item Consider some $i$-face $f$ of $\tri$ for $i < d$,
        and let $\Delta_1,\ldots,\Delta_t$ be the simplices of $\tri$
        that contain $f$ as a subface.

        Recall that $f$ is in fact an equivalence
        class of individual $i$-faces of the simplices
        $\Delta_1,\ldots,\Delta_t$ that are identified
        \emph{as a consequence of the facet gluings}.
        In other words, the individual appearances of $f$ in each simplex
        $\Delta_1,\ldots,\Delta_t$ are linked by a series of gluings
        of facets of these simplices---that is, arcs of the dual graph
        $\dual(\tri)$ that connect $\Delta_1,\ldots,\Delta_t$.

        The bags containing each $\Delta_i$ form a connected
        subtree $T_i$, and each arc of $\dual(\tri)$ that joins
        $\Delta_i$ with $\Delta_j$ has a corresponding bag
        $B'_\tau$ for which $\Delta_i,\Delta_j \in B'_\tau$.
        Therefore these arcs effectively join the subtrees
        $T_1,\ldots,T_t$ together, and so all bags containing $f$ form a
        single (larger) connected subtree of $T$.
        \qedhere
    \end{itemize}
\end{proof}

We emphasise that the proof of Lemma~\ref{l-hasse-tw} makes critical use
of our definition of a $d$-dimensional triangulation, where we only
explicitly identify \emph{facets} of $d$-simplices, and all
lower-dimensional face identifications are just a consequence of this.
The proof above does not work for more general simplicial complexes (where
lower-dimensional faces can be independently identified), but again we
note that our setting here covers all reasonable definitions of a
triangulated \emph{manifold}.

\begin{lemma} \label{l-hasse-mso}
    For any fixed dimension $d \in \N$,
    every MSO formula $\phi(x_1,\ldots,x_t)$
    on $d$-dimensional triangulations
    has a corresponding MSO formula
    $\overline{\phi}(\overline{x}_1,\ldots,\overline{x}_t)$ on
    edge-coloured graphs with $k=\sum_{i=1}^{d+1} i!$ colours
    for which,
    for any $d$-dimensional triangulation $\tri$:
    \begin{itemize}
        \item Any assignment of values to $x_1,\ldots,x_t$ for which
        $\tri \models \phi(x_1,\ldots,x_t)$ yields a corresponding
        assignment of values to $\overline{x}_1,\ldots,\overline{x}_t$
        for which
        $\hasse(\tri) \models
        \overline{\phi}(\overline{x}_1,\ldots,\overline{x}_t)$.
        This assignment is obtained by replacing faces of $\tri$
        with the corresponding nodes of $\hasse(\tri)$.

        \item Conversely, any assignment of values to
        $\overline{x}_1,\ldots,\overline{x}_t$ for which
        $\hasse(\tri) \models
        \overline{\phi}(\overline{x}_1,\ldots,\overline{x}_t)$
        is derived from an
        assignment of values to $x_1,\ldots,x_t$ for which
        $\tri \models \phi(x_1,\ldots,x_t)$, as described above.
    \end{itemize}
\end{lemma}

As with the earlier Lemma~\ref{l-simple-mso}, this in particular
implies that if $\phi$ is a sentence then
$\tri \models \phi$ if and only if $\hasse(\tri) \models \overline{\phi}$,
and if $\phi$ has free variables then the solutions
to $\tri \models \phi(x_1,\ldots,x_t)$ are in
bijection with the solutions to
$\hasse(\tri) \models \overline{\phi}(\overline{x}_1,\ldots,\overline{x}_t)$.

\begin{proof}
    Following the proof of Lemma~\ref{l-simple-mso},
    we work in three stages:
    (i)~we develop ``helper'' relations
    to constrain the roles that variables can play in $\overline{\phi}$;
    (ii)~we translate each component of $\phi$
    to a corresponding component of $\overline{\phi}$; and
    (iii)~we add additional constraints to $\overline{\phi}$ to avoid
    spurious unwanted solutions to $\hasse(\tri) \models
    \overline{\phi}(\overline{x}_1,\ldots,\overline{x}_t)$.

    For convenience, when developing the formula $\overline{\phi}$
    we give our $\sum_{i=1}^{d+1} i!$ colours the same labels that
    would appear in a coloured Hasse diagram.
    That is, we use the empty colour $-$, plus the series of colours
    $\pi_0\ldots\pi_i$ as described in Definition~\ref{d-hasse}.

    In the first stage we develop ``helper'' unary relations
    $\isface_i(x)$ for $i=0,\ldots,d$,
    using MSO logic on edge-coloured graphs.
    These relations only makes sense
    when interpreting $\overline{\phi}$ on the coloured Hasse diagram
    of a $d$-dimensional triangulation;
    in other contexts they are well-defined but meaningless.

    Each relation $\isface_i(x)$ encodes the fact that a
    variable in $\overline{\phi}$
    represents an $i$-face of the original triangulation $\tri$.
    To build these relations using MSO logic:
    \begin{itemize}
        \item We encode $\isface_0(x)$ by requiring that $x$ is incident
        with an arc of the empty colour $-$ and also an arc of
        some other colour.
        \item We encode $\isface_i(x)$ for $1 \leq i \leq d-1$ by
        requiring that $x$ is incident with an arc of some colour
        $\pi_0\ldots\pi_{i-1}$ and also an arc of some colour
        $\pi_0\ldots\pi_i$.
        \item We encode $\isface_d(x)$ by requiring that $x$ is incident
        with an arc of some colour $\pi_0\ldots\pi_{d-1}$ but no arc of any
        colour $\pi_0\ldots\pi_{d-2}$.
    \end{itemize}

    In the second stage, we translate each piece of $\phi$ to
    a piece of $\overline{\phi}$ as follows:
    \begin{itemize}
        \item Each variable $x_i$ in $\phi$ is replaced by
        $\overline{x}_i$ in $\overline{\phi}$.

        \item Standard logical operations (such as $\wedge$, $\vee$,
        $\neg$, $\rightarrow$ and so on), standard quantifiers
        ($\forall$ and $\exists$), and the equality and inclusion
        relations ($=$ and $\in$) are copied directly from $\phi$ to
        $\overline{\phi}$.

        \item Each subface relation
        $(f \leq_{\pi_0\ldots\pi_i} s)$ in $\phi$ is replaced by a
        long (but bounded length) phrase in $\overline{\phi}$
        that ensures $\isface_i(f) \wedge \isface_{d}(s)$,
        and that enforces the exact subface relation by enumerating
        all possible chains through the Hasse diagram that pass through
        intermediate $(i+1),\ldots,(d-1)$-faces.

        For example, in $d=3$ dimensions the edge-to-tetrahedron
        relationship $f^{(1)} \leq_{01} s$ could be encoded as:
        \[ \begin{array}{ll}
        \multicolumn{2}{l}{\isface_1(f^{(1)}) \wedge \isface_3(s) \wedge{}}
        \smallskip\\
        \left(\ \exists f^{(2)}\right.
                        & [\adj_{01}(f^{(1)}, f^{(2)})\ \wedge
                         \ \adj_{012}(f^{(2)}, s)]\quad\vee {}\\
                        & [\adj_{10}(f^{(1)}, f^{(2)})\ \wedge
                         \ \adj_{102}(f^{(2)}, s)]\quad\vee {}\\
                        & [\adj_{02}(f^{(1)}, f^{(2)})\ \wedge
                         \ \adj_{021}(f^{(2)}, s)]\quad\vee {}\\
                        & [\adj_{20}(f^{(1)}, f^{(2)})\ \wedge
                         \ \adj_{120}(f^{(2)}, s)]\quad\vee {}\\
                        & [\adj_{12}(f^{(1)}, f^{(2)})\ \wedge
                         \ \adj_{201}(f^{(2)}, s)]\quad\vee {}\\
                        & [\adj_{21}(f^{(1)}, f^{(2)})\ \wedge
                         \ \adj_{210}(f^{(2)}, s)]
                           \left.\llap{\phantom{$f^{(1)}$}}\qquad\right).
        \end{array} \]
    \end{itemize}

    This ensures that,
    if an assignment of values to $x_1,\ldots,x_t$ gives
    $\tri \models \phi(x_1,\ldots,x_t)$, the corresponding
    assignment of values to $\overline{x}_1,\ldots,\overline{x}_t$
    gives $\hasse(\tri) \models
    \overline{\phi}(\overline{x}_1,\ldots,\overline{x}_t)$.

    In the third stage, we eliminate additional and unwanted solutions to
    $\hasse(\tri) \models
    \overline{\phi}(\overline{x}_1,\ldots,\overline{x}_t)$
    by insisting that, for every variable $x$ that appears in $\phi$
    (either bound or free),
    the translated variable $\overline{x}$ in $\overline{\phi}$
    must represent an $i$-face or set of $i$-faces in the triangulation:
    \[ \begin{array}{l}
    \isface_0(\overline{x}) \vee \ldots \vee \isface_d(\overline{x}) \vee {}\\
       \left[\forall z\ z \in \overline{x} \rightarrow \isface_0(z)\right] \vee
       \ldots \vee
       \left[\forall z\ z \in \overline{x} \rightarrow \isface_d(z)\right].
    \end{array}\]
    This ensures that every solution to $\hasse(\tri) \models
    \overline{\phi}(\overline{x}_1,\ldots,\overline{x}_t)$
    translates back to a corresponding solution to
    $\tri \models \phi(x_1,\ldots,x_t)$.
\end{proof}

\begin{theorem}
    \label{t-tri}
    For fixed dimension $d \in \N$,
    let $K$ be any class of $d$-dimensional triangulations whose dual
    graphs have universally bounded treewidth.  Then:
    \begin{itemize}
        \item
        For any fixed MSO sentence $\phi$,
        it is possible to test whether $\tri \models \phi$
        for triangulations $\tri \in K$ in time $O(|\tri|)$.
        \item
        For any restricted MSO extremum problem $P$,
        it is possible to solve $P$ for triangulations
        $\tri \in K$ in time $O(|\tri|)$ under the uniform cost measure.
        \item
        For any MSO evaluation problem $P$,
        it is possible to solve $P$ for triangulations
        $\tri \in K$ in time $O(|\tri|)$ under the uniform cost measure.
    \end{itemize}
\end{theorem}

\begin{proof}
    The proof is essentially the same as for Theorem~\ref{t-courcelle-coloured}
    (Courcelle's theorem for edge-coloured graphs).
    Let $P$ be any such problem, and let $\phi$ be its underlying
    MSO sentence or formula.  Using Lemma~\ref{l-hasse-mso}
    we can translate $\phi$ to $\overline{\phi}$,
    yielding a new problem $\overline{P}$ on edge-coloured graphs.
    Then, for any triangulation $\tri$ given as input to $P$,
    we construct the Hasse diagram $\hasse(\tri)$ and solve
    $\overline{P}$ for $\hasse(\tri)$ instead.
    For extremum problems the size of each set (and hence the value of the
    extremum) does not change,
    and for evaluation problems we copy the input weights from $i$-faces
    of $\tri$ directly to the corresponding nodes of $\hasse(\tri)$.

    Noting that $d$ is fixed,
    Lemma~\ref{l-hasse-size} shows that
    $\hasse(\tri)$ has size $O(|\tri|)$ and can therefore be constructed
    in $O(|\tri|)$ time, and
    Lemma~\ref{l-hasse-tw} ensures that the treewidth
    $\tw(\hasse(\tri))$ is universally bounded.
    The result now follows directly from
    the edge-coloured variants of Courcelle's theorem
    (Theorem~\ref{t-courcelle-coloured}).
\end{proof}

%
%

\section{Applications} \label{s-app}

We present several applications of our main result (Theorem~\ref{t-tri});
most are in
the realm of \emph{3-manifold topology}, where algorithmic questions are
often solvable but highly complex, and parameterised complexity
is just beginning to be explored.
\begin{itemize}
    \item We recover two of the earliest fixed-parameter tractability results
    on 3-manifolds as a direct corollary of Theorem~\ref{t-tri}.
    The underlying problems are a decision problem on detecting
    taut angle structures \cite{burton13-taut},
    and an extremum problem on optimal discrete Morse matchings
    \cite{burton13-morse}.
    \item We prove two new fixed-parameter tractability results.
    These include a $d$-dimensional generalisation of the
    discrete Morse matching result,
    plus a result on computing the powerful Turaev-Viro invariants for
    3-manifolds.
\end{itemize}

Although ``treewidth of the dual graph'' might seem artificial as a parameter,
this is both natural and useful for 3-manifold
triangulations---there are many common constructions in 3-manifold topology
that are conducive to small treewidth even when the
total number of tetrahedra is large.  For instance:
\begin{itemize}
    \item Dehn fillings do not increase treewidth when performed in
    the natural way by attaching layered solid tori \cite{jaco03-0-efficiency}.
    \item The ``canonical'' triangulations of arbitrary Seifert fibred spaces
    over the sphere have a universally bounded treewidth of just
    two \cite{burton13-regina}. 
    \item Building a complex 3-manifold triangulation from smaller
    blocks with ``narrow'' $O(1)$-sized connections can also keep treewidth
    small.  See for instance JSJ decompositions \cite{jaco79-jsj-book} or the
    bricks of Martelli and Petronio \cite{martelli02-decomp},
    where each connection involves just two triangles.
\end{itemize}


\subsection{Taut angle structures}
\label{s-taut}

Given a 3-dimensional triangulation $\tri$, a \emph{taut angle structure}
assigns internal dihedral angles $0,0,0,0,\pi,\pi$ to the six edges
of each tetrahedron of $\tri$ so that:
\begin{itemize}
    \item the two $\pi$ angles are assigned to opposite edges in each
    tetrahedron (as in Figure~\ref{fig-taut-pi});
    \item for each edge $e$ of the triangulation, the sum of all angles
    assigned to $e$ is precisely $2\pi$ (as in Figure~\ref{fig-taut-sum}).
\end{itemize}

\begin{figure}[tb]
    \centering
    \subfigure[The $\pi$ angles in each tetrahedron are assigned to
        opposite edges\label{fig-taut-pi}]{%
        \qquad\qquad\includegraphics[scale=1.0]{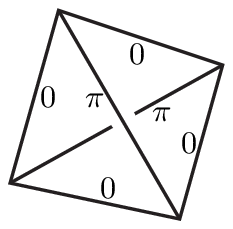}\qquad\qquad}
    \qquad
    \subfigure[The angle sum around each edge is $2\pi$\label{fig-taut-sum}]{%
        \includegraphics[scale=1.0]{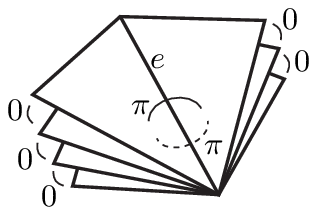}}
    \caption{Defining a taut angle structure}
    \label{fig-taut}
\end{figure}

Essentially, a taut angle structure shows how the tetrahedra can be
``squashed flat'' in a manner that is globally consistent.  They are
typically used in the context of \emph{ideal triangulations}
(where $\tri$ becomes a non-compact 3-manifold after its vertices are removed),
and they play an interesting role in linking the combinatorics of a
triangulation with the geometry and topology of the underlying 3-manifold
\cite{hodgson11-veering,lackenby00-taut}.

\begin{problem}
    \textsc{taut angle structure} is the decision problem whose
    input is a 3-dimensional triangulation $\tri$, and whose output is
    \texttt{true} or \texttt{false} according to whether there exists
    a taut angle structure on $\tri$.
\end{problem}

Na{\"i}vely, this problem can be solved using an exponential-sized
combinatorial search through all $3^{|\tri|}$ possible
assignments of angles---in each tetrahedron there are three choices for
which pair of opposite edges receive the angles $\pi$.  For small
treewidth triangulations, however, it has been shown that we can do much
better:
\begin{theorem}[B.-Spreer \cite{burton13-taut}] \label{t-taut}
    The problem \textsc{taut angle structure} is linear-time
    fixed-parameter tractable, where the parameter is the treewidth of
    the dual graph of the input triangulation.

    More specifically, for any class $K$ of 3-dimensional triangulations
    whose dual graphs have universally bounded treewidth $\leq k$,
    we can solve \textsc{taut angle structure} for any input
    triangulation $\tri \in K$ in time
    $O(3^{7k}k \cdot |\tri|)$.
\end{theorem}

We observe now that we can recover the linear-time fixed-parameter
tractability result directly from Theorem~\ref{t-tri}
(though we do not recover the precise ``constant'' $3^{7k}k$).
All we need to do is express the existence of a taut angle structure
using MSO logic.  The full MSO sentence is messy, and so we merely
outline the key ideas:
\begin{itemize}
    \item To encode the assignment of angles,
    we introduce set variables $T_1,T_2,T_3$ that partition the
    tetrahedra of the input triangulation $\tri$.
    The set $T_1$ represents the tetrahedra with
    $\pi$ angles on opposite edges 01 and 23,
    set $T_2$ represents the tetrahedra with
    $\pi$ angles on opposite edges 02 and 13, and
    set $T_3$ represents the tetrahedra with
    $\pi$ angles on opposite edges 03 and 12.

    \item To enforce the $2\pi$ angle sum criterion,
    we introduce an edge variable $f^{(1)}$ and ensure that
    $\forall f^{(1)}$, the edge $f^{(1)}$ is assigned
    the value $\pi$ precisely twice.
    Here a \emph{single} assignment of $\pi$ to $f^{(1)}$ corresponds
    to some simplex $s$ for which:
    \[ \begin{array}{l}
        s \in T_1\ \mathrm{and}\ \left(
        f^{(1)} \leq_{01} s \vee
        f^{(1)} \leq_{10} s \vee
        f^{(1)} \leq_{23} s \vee
        f^{(1)} \leq_{32} s\right);\ \mathrm{or} \\
        s \in T_2\ \mathrm{and}\ \left(
        f^{(1)} \leq_{02} s \vee
        f^{(1)} \leq_{20} s \vee
        f^{(1)} \leq_{13} s \vee
        f^{(1)} \leq_{31} s\right);\ \mathrm{or} \\
        s \in T_3\ \mathrm{and}\ \left(
        f^{(1)} \leq_{03} s \vee
        f^{(1)} \leq_{30} s \vee
        f^{(1)} \leq_{12} s \vee
        f^{(1)} \leq_{21} s\right).
    \end{array} \]
    We must piece together the full MSO clause with care, since
    $f^{(1)}$ may receive two assignments of $\pi$ from the same simplex
    (e.g., a simplex $s \in T_1$ for which both
    $f^{(1)} \leq_{01} s$ and $f^{(1)} \leq_{23} s$).
\end{itemize}

By formalising these ideas into a full MSO sentence,
the linear-time fixed-parameter tractability of
\textsc{taut angle structure}
follows directly from the main result of this paper (Theorem~\ref{t-tri}).


\subsection{Discrete Morse matchings}
\label{s-morse}

In essence, discrete Morse theory offers a way to study the
``topological complexity'' of a triangulation.  The idea is to
effectively quarantine the topological content of a triangulation
into a small number of ``critical faces''; the remainder of the
triangulation then becomes ``padding'' that is topologically unimportant.

A key problem in discrete Morse theory is to find an
\emph{optimal Morse matching}, where the number of critical faces is
as small as possible.  Solving this problem yields important topological
information \cite{dey99-computop}, and has a number of practical
applications \cite{gunther12-morse-smale,lewiner04-morse}.

We give a very brief overview of the key concepts here; see
the very accessible paper \cite{forman02-morse} for details.
Morse matchings are typically defined in terms of the Hasse diagram of a
simplicial complex.  Here we work with the coloured Hasse diagram
instead, which allows us to port the necessary concepts to the
more general $d$-dimensional triangulations
that we use in this paper.

\begin{defn}
    Let $\tri$ be a $d$-dimensional triangulation with coloured Hasse
    diagram $\hasse(\tri)$.  Let $V^{(i)}$ denote the set of nodes of
    $\hasse(\tri)$ that correspond to $i$-faces of $\tri$ for each
    $i=0,\ldots,d$, and let $E$ denote the set of arcs of $\hasse(\tri)$.

    A \emph{Morse matching} on $\tri$ is a set of arcs $M \subseteq E$
    with the following properties:
    \begin{itemize}
        \item The arcs in $M$ are \emph{disjoint} (i.e., no two are
        incident with a common node),
        and no arc in $M$ is incident with the empty node $\emptyset$.
        \item For each $i=0,\ldots,d-1$, there are
        no \emph{alternating cycles} between $V^{(i)}$ and $V^{(i+1)}$.
        That is, $\hasse(\tri)$ has no cycle
        whose nodes alternate between $V^{(i)}$ and $V^{(i+1)}$,
        and whose arcs alternate between $M$ and $E \backslash M$.
    \end{itemize}

    For a given Morse matching $M$, a \emph{critical $i$-face} of $\tri$
    is an $i$-face whose corresponding node $v \in V^{(i)}$ is not incident
    with any arc $e \in M$.  We let $c(M)$ denote the total number of
    critical faces, so $c(M) = \sum_{i=0}^d |V^{(i)}| - 2|M|$.
\end{defn}

\begin{example}
    Consider the coloured Hasse diagram seen earlier in
    Figure~\ref{fig-hasse}.  This diagram has several pairs of parallel arcs,
    none of which can be used in a Morse matching (since this would
    yield an alternating cycle of length two).  Therefore the largest
    Morse matching possible has just $|M|=1$ arc (either the arc from
    $e$ to $\Delta_1$, or the arc from $e$ to $\Delta_2$).
    Such a matching has $c(M)=4$ critical faces (one 2-face, two 1-faces,
    and one 0-face).
\end{example}

\begin{problem}
    \textsc{$d$-dimensional optimal Morse matching} is the
    extremum problem whose input is a $d$-dimensional triangulation $\tri$,
    and whose output is the minimum of $c(M)$ over all Morse matchings
    $M$ on $\tri$.
\end{problem}

Even in dimension $d=3$ the related decision problem
is NP-complete \cite{joswig06-morse},
but again for small treewidth triangulations we can do better:

\begin{theorem}[B.-Lewiner-Paix{\~a}o-Spreer \cite{burton13-morse}]
    The problem \textsc{3-dimensional optimal Morse matching} is linear-time
    fixed-parameter tractable, where the parameter is the treewidth of
    the dual graph of the input triangulation.

    More specifically, for any class $K$ of 3-dimensional triangulations
    whose dual graphs have universally bounded treewidth $\leq k$,
    we can solve the problem for any input
    triangulation $\tri \in K$ in time
    $O(4^{k^2+k} \cdot k^3 \cdot \log k \cdot |\tri|)$.
\end{theorem}

Again we show now that we can recover the
linear-time fixed-parameter tractability result directly from
Theorem~\ref{t-tri} (but not the precise ``constant''
$4^{k^2+k} \cdot k^3 \cdot \log k$).
Moreover, we generalise this to arbitrary dimensions:

\begin{theorem} \label{t-morse}
    For fixed dimension $d \in \N$ and
    any class $K$ of $d$-dimensional triangulations whose dual
    graphs have universally bounded treewidth,
    we can solve \textsc{$d$-dimensional optimal Morse matching}
    for triangulations $\tri \in K$ in time $O(\tri)$
    under the uniform cost measure.
\end{theorem}

\begin{proof}
    To formulate this as an MSO extremum problem on triangulations:
    \begin{itemize}
        \item For each $i=0,\ldots,d$ we introduce a set variable
        $V^{(i)}$ and a clause to ensure that $V^{(i)}$ contains
        all $i$-faces of $\tri$.  These trivial set variables
        are used in the rational function that we seek to minimise
        (see below).
        \item To encode a Morse matching, we use a family of set
        variables $W^{(i)}_{\pi_0\ldots\pi_{i-1}}$
        for each $i=1,\ldots,d$ and for each sequence
        $\pi_0\ldots\pi_{i-1}$ of distinct integers in $\{0,\ldots,i\}$.
        Each variable $W^{(i)}_{\pi_0\ldots\pi_{i-1}}$ holds all
        $i$-faces of $\tri$ that are matched with an
        $(i-1)$-face using an arc of $\hasse(\tri)$ of colour
        $\pi_0\ldots\pi_{i-1}$.
        \item For each set $W^{(i)}_{\pi_0\ldots\pi_{i-1}}$,
        we add a clause to ensure that the corresponding arcs
        exist in $\hasse(\tri)$:
        $\forall v^{(i)}\ \exists v^{(i-1)}
        \ v^{(i)} \in W^{(i)}_{\pi_0\ldots\pi_{i-1}} \rightarrow
        v^{(i-1)} \leq_{\pi_0\ldots\pi_{i-1}} v^{(i)}$.
        \item We add a (long) clause to ensure that each face of $\tri$
        is involved in at most one matching.
        \item For each $i=1,\ldots,d$, we add a (very long) clause
        to ensure that there are no alternating cycles between
        $i$-faces and $(i-1)$-faces.
        Here we encode an alternating cycle using a second family of set
        variables $X^{(i)}_{\pi_0\ldots\pi_{i-1}}$ that represent arcs
        present in the cycle but not the matching, where:
        \begin{itemize}
        \item for each $i$ and each $\pi_0\ldots\pi_{i-1}$ the sets
        $W^{(i)}_{\pi_0\ldots\pi_{i-1}}$ and $X^{(i)}_{\pi_0\ldots\pi_{i-1}}$
        are disjoint;
        \item each $i$-face $v^{(i)}$ appears in either none of the sets
        $W^{(i)}_{\pi_0\ldots\pi_{i-1}}$ or
        $X^{(i)}_{\pi_0\ldots\pi_{i-1}}$, or else exactly one of the
        sets $W^{(i)}_{\pi_0\ldots\pi_{i-1}}$ and one of the sets
        $X^{(i)}_{\pi_0\ldots\pi_{i-1}}$;
        \item each $(i-1)$-face $v^{(i-1)}$ is involved in precisely
        zero or two relationships of the form
        $v^{(i-1)} \leq_{\pi_0\ldots\pi_{i-1}} v^{(i)}$
        where
        $v^{(i)} \in W^{(i)}_{\pi_0\ldots\pi_{i-1}} \cup
        X^{(i)}_{\pi_0\ldots\pi_{i-1}}$.
        \end{itemize}
    \end{itemize}

    We can now encode \textsc{$d$-dimensional optimal Morse matching}
    as the extremum problem that asks to minimise
    $\sum_i |V^{(i)}| - 2 \sum_i |W^{(i)}_{\pi_0\ldots\pi_{i-1}}|$,
    and the result follows from Theorem~\ref{t-tri}.
\end{proof}


\subsection{The Turaev-Viro invariants}
\label{s-tv}

The Turaev-Viro invariants are an infinite family of topological
invariants of 3-manifolds \cite{turaev92-invariants}.
For every triangulation $\tri$ of a closed 3-manifold, there is an invariant
$|\tri|_{r,q_0}$ for each integer $r \geq 3$
and each $q_0 \in \C$ for which $q_0$ is a $(2r)$th root of unity
and $q_0^2$ is a primitive $r$th root of unity.
A key property (which justifies the name ``invariants'') is that
the value of $|\tri|_{r,q_0}$ depends only upon the topology of
the underlying 3-manifold, and not the particular choice of
triangulation $\tri$.


The Turaev-Viro invariants can be expressed as sums over combinatorial
objects on $\tri$, and so (unlike many other 3-manifold invariants)
lend themselves well to computation.
Moreover, they have proven extremely useful in practical software settings
for distinguishing between different 3-manifolds
\cite{burton07-nor7,matveev03-algms}.
However, they have a major drawback: computing $|\tri|_{r,q_0}$
requires time $O(r^{2|\tri|} \times \mathrm{poly}(|\tri|))$
under existing algorithms, and so is feasible
only for small triangulations and/or small $r$.

In Theorem~\ref{t-tv} we show again that
we can do much better for small treewidth triangulations:
for fixed $r$,
computing $|\tri|_{r,q_0}$ is linear-time fixed-parameter tractable,
with (as usual) the treewidth of the dual graph as the parameter.

Before proving this result, we give a short outline of how the
Turaev-Viro invariants $|\tri|_{r,q_0}$ are defined.
The formula is relatively detailed and so
our summary here is very brief, with just enough information to prove
the subsequent theorem.  For more information, including
motivations and topological context for these invariants,
we refer the reader to references
such as \cite{matveev03-algms,turaev92-invariants}.

\begin{defn} \label{d-tv}
    Let $\tri$ be a closed 3-manifold triangulation,
    and let $V$, $E$ and $S$ denote the set of all vertices, edges and
    tetrahedra of $\tri$.
    Let $r \geq 3$ be an integer,
    and let $q_0 \in \C$ be a $(2r)$th root of unity
    for which $q_0^2$ is a primitive $r$th root of unity.
    We define the Turaev-Viro invariant $|\tri|_{r,q_0}$ as follows.

    Let $I = \{0,1/2,1,3/2,\ldots,(r-2)/2\}$, so $|I|=r-1$.
    A triple $(i,j,k) \in I^3$ is called \emph{admissible} if
    $i+j+k \in \Z$, $i \leq j+k$, $j \leq i+k$, $k \leq i+j$,
    and $i+j+k \leq r-2$.

    A \emph{colouring} $\theta$ of $\tri$ is a map $\theta\co E \to I$;
    in other words, we ``colour'' each edge of $\tri$ with an element of $I$.
    A colouring $\theta$ is called \emph{admissible}
    if, for each 2-face $f$ of $\tri$, the three edges
    $e_1,e_2,e_3 \in E$ bounding $f$ give rise to an admissible
    triple $(\theta(e_1),\theta(e_2),\theta(e_3))$.

    We make use of several complex constants that depend on $r$ and $q_0$.
    We do not need to give their exact values here;
    details can be found in \cite{turaev92-invariants}.
    These constants are:
    $\alpha \in \C$; $\beta_i \in \C$ for each $i \in I$;
    and $\gamma_{i,j,k,\ell,m,n} \in \C$ for each $i,j,k,\ell,m,n \in I$.
    Note that these constants do not depend
    upon the specific triangulation $\tri$.

    We use these constants as follows.
    Let $\theta$ be an admissible colouring of $\tri$.
    For each vertex $v \in V$ we define $|v|_\theta=\alpha$.
    For each edge $e \in E$ we define $|e|_\theta = \beta_{\theta(e)}$.
    For each tetrahedron $\Delta \in S$ we define
    $|\Delta|_\theta =
    \gamma_{\theta(e_{01}),\theta(e_{02}),\theta(e_{12}),
            \theta(e_{23}),\theta(e_{13}),\theta(e_{03})}$,
    where $e_{ij} \in E$ is the edge of the triangulation that
    joins vertices $i$ and $j$ of $\Delta$.
    For the entire triangulation we define
    \( |\tri|_\theta =
        \left(\prod_{v \in V} |v|_\theta\right)
        \cdot
        \left(\prod_{e \in E} |e|_\theta\right)
        \cdot
        \left(\prod_{\Delta \in S} |\Delta|_\theta\right) \).

    Finally, we define the Turaev-Viro invariant
    $|\tri|_{r,q_0} = \sum_\theta |\tri|_\theta$, where this sum is taken
    over all admissible colourings $\theta$ of $\tri$.
\end{defn}

\begin{theorem} \label{t-tv}
    For any fixed integer $r \geq 3$
    and any class $K$ of closed $3$-manifold triangulations
    whose dual graphs have universally bounded treewidth,
    we can compute any Turaev-Viro invariant
    $|\tri|_{r,q_0}$ for any closed 3-manifold triangulation
    $\tri \in K$ in time $O(\tri)$
    under the uniform cost measure.
\end{theorem}

\begin{proof}
    We prove this by framing the computation of
    $|\tri|_{r,q_0}$ as a multiplicative MSO evaluation problem.

    For our MSO formula $\phi$, we define several free variables that together
    encode an admissible colouring $\theta$ of $\tri$.
    These include:
    \begin{itemize}
        \item A single set variable $V$ that stores all vertices of $\tri$.
        \item A set variable $E_i$ for each $i \in I$.
        Each set $E_i$ represents those edges of $\tri$
        that are assigned the colour $i$.
        \item A set variable $S_{i,j,k,\ell,m,n}$ for each
        $i,j,k,\ell,m,n \in I$ where the triples
        $(i,j,k)$, $(k,\ell,m)$, $(m,n,i)$ and $(j,\ell,n)$ are
        all admissible.
        Each set $S_{i,j,k,\ell,m,n}$ represents those tetrahedra whose
        edges $01$, $02$, $12$, $23$, $13$ and $03$ are assigned colours
        $i,j,k,\ell,m,n$ respectively, which means that every 2-face of
        such a tetrahedron will be coloured by an admissible triple.
    \end{itemize}
    We insert clauses into our MSO formula to ensure that:
    \begin{itemize}
        \item The set variable $V$ contains precisely all vertices of $\tri$.
        \item The set variables $E_i$ together partition the edges of $\tri$,
        and the set variables $S_{i,j,k,\ell,m,n}$ together partition the
        tetrahedra of $\tri$.
        \item The colourings $\{E_i\}$ and $\{S_{i,j,k,\ell,m,n}\}$ are
        consistent.  This involves clauses
        $\forall s\forall e^{(1)}
         \ [s \in S_{i,j,k,\ell,m,n} \wedge
           (e^{(1)} \leq_{01} s \vee e^{(1)} \leq_{10} s)] \rightarrow
           e^{(1)} \in E_i$
        and many others of a similar form, where $e^{(1)}$ is an edge
        variable and $s$ is a tetrahedron variable.
    \end{itemize}

    It follows that the solutions to $\tri \models \phi$
    correspond precisely to the admissible colourings of $\tri$.

    We now assign weights for our evaluation problem as follows:
    \begin{itemize}
        \item For the set variable $V$, we assign a weight of
        $\alpha$ to each vertex of $\tri$,
        and a weight of $1$ to all other faces of all other dimensions.
        \item For each set variable $E_i$, we assign a weight of
        $\beta_i$ to each edge of $\tri$,
        and a weight of $1$ to all other faces of all other dimensions.
        \item For each set variable $S_{i,j,k,\ell,m,n}$, we assign a weight of
        $\gamma_{i,j,k,\ell,m,n}$ to each tetrahedron of $\tri$,
        and a weight of $1$ to all other faces of all other dimensions.
    \end{itemize}

    By Definition~\ref{d-tv},
    the Turaev-Viro invariant $|\tri|_{r,q_0}$ is precisely the
    solution to the multiplicative
    evaluation problem defined by this MSO formula and these weights.
    Note that, whilst the MSO formula depends only on $r$
    (which is fixed), the weights depend on both $r$ and $q_0$
    (where $q_0$ is supplied with the input).

    The result now follows directly from Theorem~\ref{t-tri}.
\end{proof}

%
%

\bibliographystyle{amsplain}
\bibliography{pure}

\end{document}